\documentclass{amsart}
\usepackage{amssymb,amsmath,latexsym}
\usepackage{amsthm}
\usepackage{fontenc}
\usepackage{amssymb}

\numberwithin{equation}{section}

\newtheorem{theorem}{Theorem}[section]
\newtheorem{corollary}{Corollary}[theorem]
\newtheorem{lemma}[theorem]{Lemma}

\begin{document}
\author{Alexander E Patkowski}
\title{On Davenport Expansions, Popov's formula, and Fine's query}

\maketitle
\begin{abstract} We establish an explicit connection between a Davenport expansion and the Popov sum. Asymptotic analysis follows as a result of these formulas. New solutions to a query of N.J. Fine are offered, and a proof of Davenport expansions is detailed. \end{abstract}



\section{Introduction and Main formulas} Let $\Lambda(n)$ denote the von Mangoldt function, $\zeta(s)$ the Riemann zeta function, and $\rho$ the non-trivial (complex) zeros of the Riemann zeta function [5, p. 43]. In a recent paper [10] we established a proof of Popov's formula [12]:
\begin{equation}\sum_{n>x}\frac{\Lambda(n)}{n^2}\left(\{\frac{n}{x}\}-\{\frac{n}{x}\}^2\right)=\frac{2-\log(2\pi)}{x}+\sum_{\rho}\frac{x^{\rho-2}}{\rho(\rho-1)}+\sum_{k\ge1}\frac{k+1-2k\zeta(2k+1)}{2k(k+1)(2k+1)}x^{-2k-2},\end{equation}
for $x>1.$ Here $\{x\}$ is the fractional part of $x,$ sometimes written as $\{x\}=x-\lfloor{x\rfloor},$ where $\lfloor{x\rfloor}$ is the floor function. The proof relied on the Mellin transform formula
\begin{equation}\frac{1}{2}\left(\{x\}^2-\{x\}\right)=\frac{1}{2\pi i}\int_{(b)}\left(\frac{s+1}{2s(s-1)}-\frac{\zeta(s)}{s}\right)\frac{x^{s+1}}{s+1}ds,\end{equation} which is valid for $x>0,$ $-1<b<0.$ This can be observed by noting (1.2) is known [15, pg. 14, eq. (2.1.4)] for $x>1,$ and by [10, pg.405] for $0<x<1,$
$$\frac{1}{2\pi i}\int_{(b)}\frac{x^{s}}{s(s-1)}ds=0.$$
Let $f(n)$ be a suitable arithmetic function such that $L(s)=\sum_{n\ge1}f(n)n^{-s}$ analytic for $\Re(s)>1,$ and $S(x)=\sum_{n\le x}f(n).$ Examining the right hand side of (1.2), it is not difficult to see that \begin{equation}\frac{1}{2}\sum_{n>x}\frac{f(n)}{n^2}\left(\{\frac{n}{x}\}-\{\frac{n}{x}\}^2\right) = \frac{1}{2\pi i}\int_{(b)}\frac{x^{-s-1}}{2s(s-1)}L(1-s)ds-\frac{1}{2\pi i}\int_{(b)}\frac{x^{-s-1}\zeta(s)}{s(s+1)}L(1-s)ds.\end{equation}
If we replace $s$ by $1-s$ in the first integral in (1.3) and apply Fubini's theorem to interchange the integrals, we see that this integral is equal to
\begin{equation} \frac{1}{2x}\int_{0}^{x} \left( \frac{1}{2\pi i}\int_{(1-b)}\frac{y^{s-2}}{s}L(s)ds \right) dy=\frac{1}{2x}\int_{0}^{x}\frac{S(y)dy}{y^2}, \end{equation}
by the Mellin-Perron formula.
If we generalize the formula of Segal [14, eq.(5)], we can see that the second integral on the right side of (1.3) is 
\begin{equation}\begin{aligned}\int_{0}^{\frac{1}{x}}\left(\sum_{n\ge1}\frac{f(n)}{n}(\{ny\}-\frac{1}{2})\right)dy
&=\int_{0}^{\frac{1}{x}}\left(-\frac{1}{\pi}\sum_{n\ge1}\frac{F(n)}{n}\sin(2\pi ny)\right)dy\\
&=\frac{1}{2\pi^2}\sum_{n\ge1}\frac{F(n)}{n^2}\left(\cos(\frac{2\pi n}{ x})-1\right), \end{aligned}\end{equation}
assuming uniform convergence where $F(n)=\sum_{d|n}f(d).$ Collecting (1.3), (1.4), and (1.5), we obtain the following theorem upon noting that Davenport's proof of uniform convergence is dependent on a special estimate [4].
\begin{theorem} Let $f(n)$ be chosen such that $L(s)$ is analytic for $\Re(s)>1,$ and that $\sum_{n\le N}f(n)e^{2\pi inx}=O(x(\log(x)^{-h}),$ for any fixed $h.$ We have for $x>1,$
$$\frac{1}{2}\sum_{n>x}\frac{f(n)}{n^2}\left(\{\frac{n}{x}\}-\{\frac{n}{x}\}^2\right)=\frac{1}{2x}\int_{0}^{x}S(y)\frac{dy}{y^2}+\frac{1}{2\pi^2}\sum_{n\ge1}\frac{F(n)}{n^2}(\cos(\frac{2\pi n}{ x})-1).$$
\end{theorem}
We have therefore proven the connection between Popov's formula and the integral $\int_{0}^{x}\frac{S(y)dy}{y^2}$ alluded to in [11] (see also [4, pg.69]). Perhaps even more fascinating is the connection to Davenport expansions [3, 8] through the sum on the far right hand side of (1.5). This Fourier series is known to be connected to the periodic Bernoulli polynomial $B_2(\{x\})-B_2=(\{x\}^2-\{x\}).$ For relevant material on Davenport expansions connected to Bernoulli polynomials, see [2, 9]. \par Recall that $h(x)\sim g(x)$ means that $\lim_{x\rightarrow\infty}\frac{h(x)}{g(x)}=1.$ Letting $x\rightarrow\infty$ and applying L'Hopsital's rule to Theorem 1, and the Residue Theorem to (1.3)--(1.4), we have the following.

\begin{corollary}  Let $f(n)$ be chosen such that $L(s)$ is analytic for $\Re(s)>1.$ Suppose that $S(x)\sim \Delta(x)$ as $x\rightarrow\infty.$ Then
$$\frac{1}{2}\sum_{n>x}\frac{f(n)}{n^2}\left(\{\frac{n}{x}\}-\{\frac{n}{x}\}^2\right)-\frac{1}{2\pi^2}\sum_{n\ge1}\frac{F(n)}{n^2}(\cos(\frac{2\pi n}{ x})-1) \sim\frac{\Delta(x)}{x^{2}},$$

as $x\rightarrow\infty.$
\end{corollary}
Notice that the sum on the left hand side of (1.1) is $\sim \frac{1}{x},$ which corresponds to the Prime Number theorem $\sum_{n\le x}\Lambda(n)\sim x$ when coupled with our corollary. \par The integral $\int_{0}^{x}\frac{S(y)dy}{y^2}$ has appeared in many recent works in the analytic theory of numbers. Namely, in the case of the von Mangoldt function $S(x)=\psi(x)=\sum_{n\le x}\Lambda(n),$ see [13], where we find a study of the function
$$\sum_{n\le x}\frac{\Lambda(n)}{n}-\frac{\sum_{n\le x}\Lambda(n)}{x}=\int_{1}^{x}\frac{\psi(y)dy}{y^2}.$$ For the case of the M$\ddot{o}$bius function (i.e. $S(x)=M(x)$ the Mertens function [15, pg.370]), Inoue [7, Corollary 3, $k=2$] gave, under the assumption of the weak Mertens Hypothesis,
$$\int_{1}^{x}\frac{M(y)dy}{y^2}=x^{-\frac{1}{2}}\sum_{\rho}\frac{x^{i\gamma}}{\zeta'(\rho)\rho(\rho-1)}+A(2)+O(x^{-1}).$$
Here $A(2)$ is a constant, and $g(x)=O(h(x))$ means $|g(x)|\le c_1h(x),$ $c_1>0$ a constant. \par We mention there is another form of the Fourier series on the far right side of Theorem 1.1. Note that [15, pg.14, eq.(2.1.5)]
\begin{equation}\{x\}=-\frac{1}{2\pi i}\int_{(c)}\frac{\zeta(s)}{s}x^{s}ds,\end{equation} where $x>0,$ and $0<c<1.$ Integrating, we have that
\begin{equation}\frac{1}{2}\left(\{x\}^2+\lfloor{x\rfloor}\right)=-\frac{1}{2\pi i}\int_{(c)}\frac{\zeta(s)}{s(s+1)}x^{s+1}ds.\end{equation}
Dividing by $x,$ computing the residue at the pole $s=0,$ and inverting the desired series in (1.7), we have
\begin{equation}\sum_{n\ge1}\frac{f(n)}{n}\left(\frac{1}{x2n}\left(\{nx\}^2+\lfloor{nx\rfloor}\right)-\frac{1}{2}\right)=-\frac{1}{2\pi i}\int_{(c-1)}\frac{\zeta(s)}{s(s+1)}x^{s}L(1-s)ds.\end{equation}
Here we used the fact that $\zeta(0)=-\frac{1}{2}.$ Hence, after comparing with our computation (1.5), we have proven the following result.
\begin{theorem} For $x>0,$
$$\sum_{n\ge1}\frac{f(n)}{n}\left(\frac{1}{x2n}\left(\{nx\}^2+\lfloor{nx\rfloor}\right)-\frac{1}{2}\right)=\frac{1}{2x\pi^2}\sum_{n\ge1}\frac{F(n)}{n^2}(\cos(2\pi n x)-1).$$

\end{theorem}
\section{Solution to The N.J. Fine query}
In [1], a positive answer was presented to a query of N.J. Fine, who asked for a continuous function $\varphi(x)$ on $\mathbb{R},$ with period $1,$ $\varphi(x)\neq-\varphi(-x),$ and 
\begin{equation}\sum_{N\ge k\ge1}\varphi(\frac{k}{N})=0.\end{equation} Namely, they gave the solutions
\begin{equation}\sum_{n\ge1}\frac{f(n)}{n}\cos(2\pi n x),\end{equation}
where $f(n)$ is chosen as the M$\ddot{o}$bius function $\mu(n)$ and the Liouville function $\lambda(n),$ [14]. Their proof utilizes a Ramanujan sum [15, pg.10]
\begin{equation} \sum_{N\ge k\ge1}\cos(\frac{2\pi k n}{N}),\end{equation} which is $N$ if $n\equiv0\pmod{N}$ and $0$ otherwise. It is also dependent on $\sum_{n\ge1}f(n)/n=0.$ In fact, it is possible to further generalize their result using these properties, which we offer in the following. 

\begin{theorem} Suppose $f(n)$ is a multiplicative arithmetic function chosen such that $\sum_{n\ge1}f(n)/n=0.$ Then 
\begin{equation}\sum_{n\ge1}\frac{f(n)}{n}\cos^m(\pi n x),\end{equation}
and
\begin{equation}\sum_{n\ge1}\frac{f(n)}{n}\sin^{2m}(\pi n x),\end{equation}
for each positive integer $m$ satisfy the properties in Fine's query.
\end{theorem}
\begin{proof} We will use [6, pg.31, section 1.320, no.5 and 7] for (2.4) and (2.5) to obtain our $\varphi(x).$ Namely
\begin{equation} \cos^{2m}(x)=\frac{1}{2^{2m}}\left(\sum_{m-1\ge k\ge0}2\binom{2m}{k}\cos(2(m-k)x)+\binom{2m}{m}\right),\end{equation}
\begin{equation} \cos^{2m-1}(x)=\frac{1}{2^{2m-2}}\sum_{m-1\ge k\ge0}\binom{2m}{k}\cos((2m-2k-1)x).\end{equation}
Putting $x=\frac{2\pi l n}{N},$ and summing over $N\ge l\ge1$ we see that the sum is a linear combination of zeros and $N$'s depending on weather $n(m-k)|N.$ In the case where $n(m-k)|N,$ we are left with a linear combination of terms which are independent of $n.$ The last term is simply $N\frac{1}{2^{2m}}\binom{2m}{m}.$ Therefore, summing over $n$ gives the result upon invoking $\sum_{n\ge1}f(n)/n=0.$ Similar arguments apply to (2.7).
Since (2.5) follows in the same way using another formula from [6, pg.31, section 1.320, no.1], we leave the details to the reader.
\end{proof}

\par We were interested finding more solutions to Fine's query by constructing a special arithmetic function with the possible property $\sum_{n\ge1}f(n)/n\neq0.$ Define \begin{equation}\chi_{m,l}^{\pm}(n):=\begin{cases} \pm(m^{l}\mp1),& \text {if } n=0\pmod{m},\\ 1, & \text{if } n\neq0\pmod{m}.\end{cases}\end{equation}
If $f(n)$ is completely multiplicative, this tells us that $$\sum_{n\ge1}\frac{\chi_{m,l}^{-}(n)f(n)}{n^{s}}=L(s)-m^{l}\sum_{n\equiv0\pmod{m}}\frac{f(n)}{n^{s}}=(1-f(m)m^{l-s})L(s).$$
\bf Definition: \rm A function is said to be of the class $\aleph$ if: (i) it is continuous on $\mathbb{R},$ (ii) is $1$-periodic (iii) is not odd, and (iv) satisfies (2.1) for each $N$ coprime to $m.$

\begin{theorem}  Suppose that $L(s)$ is analytic for $\Re(s)>1.$ For natural numbers $m>1,$ $l>1,$ and $N$ is coprime to $m,$ we have $D_1(x)\in\aleph$ where
\begin{equation}D_1(x)=\sum_{n\ge1}\frac{\chi_{m,l}^{+}(n)f(n)}{n^{l}}\cos(2\pi n x),\end{equation}
for a completely multiplicative function with the property $f(m)=-1,$ and $D_2(x)\in\aleph$ where
\begin{equation}D_2(x)=\sum_{n\ge1}\frac{\chi_{m,l}^{-}(n)f(n)}{n^l}\cos(2\pi n x),\end{equation}
for a completely multiplicative function with the property $f(m)=1.$
\end{theorem}

\begin{proof} First we consider (2.9). Note that because $\chi_{m,l}^{\pm}(n)$ is not completely multiplicative, we need to restrict $N$ to be coprime to $m,$ since then $\chi_{m,l}^{\pm}(Nn)=\chi_{m,l}^{\pm}(n).$ That is to say that $Nn\equiv0\pmod{m}$ is solved by $n\equiv0\pmod{m}$ provided that $N$ is coprime to $m.$ The same argument applies in the case $Nn\not\equiv0\pmod{m}.$ Using (2.3) and the method applied in [1] we compute that 
$$\begin{aligned}\sum_{N\ge k \ge1}\sum_{n\ge1}\frac{\chi_{m,l}^{+}(n)f(n)}{n^{l}}\cos(\frac{2\pi n k}{N}) &=N\sum_{n\equiv0\mod{N}}\frac{\chi_{m,l}^{+}(n)f(n)}{n^{l}} \\ 
&=\frac{f(N)}{N^{l-1}}\lim_{s\rightarrow l}(1-m^{l-s})L(s)=0. \end{aligned}.$$
The computation for (2.10) is similar, and so we leave the details for the reader.
\end{proof}
An example for $D_2(x)$ is if $f(n)=\lambda(n),$ and $m=4,$ $l>1,$ since $\lambda(4)=1.$ One for $D_1(x)$ is if $f(n)=\mu(n),$ and $m=5,$ $l>1,$ since $\mu(5)=-1.$

\section{Fourier Analysis of Davenport expansions}
In [8, pg.280--281], it is noted that Davenport's expansion may be obtained from standard Fourier techniques, and evaluating a Fourier integral involving the fractional part function. Here we will give a detailed proof to obtain a further expansion.
\begin{lemma} Let $\phi(x)$ be a $1$-periodic function on $[0,1].$ Then $\phi(x)$ admits the expansion
$$\phi(x)=\sum_{n\ge1}c_n\sin(\pi n x) \cos(\pi nx),$$
where $$c_n=8\int_{0}^{1}\phi(y)\sin(\pi n y)\cos(\pi n y)dy.$$
\end{lemma}
\begin{proof} A computation gives us
$$\begin{aligned}\int_{0}^{1}\sin(\pi n y)\cos(\pi n y)\sin(\pi m y)\cos(\pi m y)dy &= \frac{1}{16\pi}\left(\frac{\sin(2\pi(m-n))}{m-n}-\frac{\sin(2\pi(m+n))}{m+n} \right)\\ &=\begin{cases} \frac{1}{8},& \text {if } n=m,\\ 0, & \text{if } n\neq m.\end{cases}\end{aligned}$$
Hence, provided $\phi(x)$ satisfies the hypothesis of the Lemma, we find the result follows.
\end{proof}
We also require a formula to evaluate integrals involving the fractional function.
\begin{lemma} ([15, pg.13] ) Suppose $\phi(x)$ has a continuous derivative in $[a,b].$ Then we have,
\begin{equation}\sum_{a<n\le b}\phi(n)=\int_{a}^{b}\phi(y)dy+\int_{a}^{b}(\{y\}-\frac{1}{2})\phi'(y)dy+(\{a\}-\frac{1}{2})\phi(a)-(\{b\}-\frac{1}{2})\phi(b). \end{equation}
\end{lemma}
Noting that $\sin(2x)=2\sin(x)\cos(x),$ it follows that we should have the following Davenport expansion.
\begin{theorem} We have,
$$\sum_{n\ge1}\frac{f(n)}{n}(\{nx\}-\frac{1}{2})=2\sum_{n\ge1}\bar{c}_n\sin(\pi n x) \cos(\pi nx),$$
where
$$\bar{c}_n= -\frac{1}{n\pi}\sum_{d|n}f(d).$$
\end{theorem}
\begin{proof} In Lemma 3.2, we put $a=0,$ $b=N,$ and set $\phi(x)=\sin(\frac{2\pi x m}{N})^2.$ In this case we have
\begin{equation}\begin{aligned}\sum_{0<n\le N}\sin(\frac{\pi n m}{N})^2 &=\frac{N}{2}+\frac{2\pi m}{N}\int_{0}^{N}(\{y\}-\frac{1}{2})\sin(\frac{\pi y m}{N})\cos(\frac{\pi y m}{N})dy \\
&=\frac{N}{2}+2\pi m\int_{0}^{1}(\{yN\}-\frac{1}{2})\sin(\pi y m)\cos(\pi y m)dy
 \end{aligned}\end{equation}
The sum on the left side of (3.2) may be evaluated in the same way as (2.3) to find that 
\begin{equation} \sum_{0<n\le N}\sin(\frac{\pi n m}{N})^2=\begin{cases} 0, & \text {if } m\equiv0\pmod{N},\\ \frac{N}{2}, & \text{} otherwise.\end{cases}\end{equation}
Collecting (3.2) and (3.3), it follows that
\begin{equation}\int_{0}^{1}(\{yN\}-\frac{1}{2})\sin(\pi y m)\cos(\pi y m)dy=\begin{cases} -\frac{N}{4\pi m}, & \text {if } m\equiv0\pmod{N},\\ 0, & \text{} otherwise.\end{cases}\end{equation}
Summing over the desired series in (3.4) for $m$ and $N$ gives the result.

\end{proof}

1390 Bumps River Rd. \\*
Centerville, MA
02632 \\*
USA \\*
E-mail: alexpatk@hotmail.com, alexepatkowski@gmail.com

\end{document}